\documentclass[12pt, a4paper]{article}

\usepackage{amssymb,latexsym,amsmath,amsthm,amsfonts,graphicx,hyperref,epsfig,array,tikz,enumerate}

\usetikzlibrary{graphs,graphs.standard}

\hypersetup{colorlinks=true, linkcolor=blue}
\usepackage[width=160mm,top=28mm,bottom=28mm]{geometry}

\newtheorem{lemma}{Lemma}[section]
\newtheorem{theorem}[lemma]{Theorem}

\newtheorem{corollary}[lemma]{Corollary}
\newtheorem{proposition}[lemma]{Proposition}
\newtheorem{remark}[lemma]{Remark}

\def\Aut{\mbox{\rm Aut}}
\def\diam{\mbox{\rm diam}}

\begin{document}
\title{\bf Proper divisor graph of a positive integer}
\author{Hitesh Kumar\footnote{The author is supported by Scholarship for Higher Education (SHE) under Innovation in Science Pursuit for Inspired Research (INSPIRE) programme of the Department of Science and Technology, Government of India.} \and Kamal Lochan Patra\and Binod Kumar Sahoo\footnote{The author is partially supported by Project No. MTR/2017/000372 of the Science and Engineering Research Board (SERB), Department of Science and Technology, Government of India.}}
\maketitle

\begin{abstract}
The proper divisor graph $\Upsilon_n$ of a positive integer $n$ is the simple graph whose vertices are the proper divisors of $n$, and in which two distinct vertices $u, v$ are adjacent if and only if $n$ divides $uv$. The graph $\Upsilon_n$ plays an important role in the study of the zero divisor graph of the ring $\mathbb{Z}_n$. In this paper, we study some graph theoretic properties of $\Upsilon_n$ and determine the graph parameters such as clique number, chromatic number, chromatic index, independence number, matching number, domination number, vertex and edge covering numbers of $\Upsilon_n$. We also determine the automorphism group of $\Upsilon_n$.
\end{abstract}

\noindent{\bf Keywords:} Graph parameters, automorphism, vertex degree, pendant vertex, cut vertex\\
{\bf AMS 2010 subject classification:} 05C07, 05C15, 05C25, 05C69, 05C70

\section{Introduction}

All graphs considered in this paper are finite and simple. We refer to the book \cite{west} for unexplained graph terminology and the basics on graph theory.
Let $G$ be a graph with vertex set $V(G)$. If $V(G)=\Phi$, then $G$ is called the {\it empty graph}. If $V(G)\neq \Phi$ but $G$ has no edge, then $G$ is called a {\it null graph}. We write $u\sim v$ for two distinct vertices $u,v\in V(G)$ if they are adjacent in $G$. The {\it degree} of a vertex $v\in V(G)$ is denoted by $\deg(v)$. A vertex of degree one is called a {\it pendant vertex} of $G$. The automorphism group of $G$ is denoted by $\Aut(G)$. A vertex $v\in V(G)$ is called a {\it cut vertex} if the induced subgraph $G-v$ of $G$ with vertex set $V(G-v)=V(G)\setminus \{v\}$ has more components than $G$.

A {\it clique} in $G$ is a subset $K$ of $V(G)$ such that the induced subgraph on $K$ is complete. The maximum size of a clique in $G$, denoted by $\omega(G)$, is called the {\it clique number} of $G$. A {\it vertex coloring} of $G$ is an assignment of colors to the vertices of $G$ such that two adjacent vertices receive different colors. The {\it chromatic number} of $G$, denoted by $\chi(G)$, is the minimum number of colors required for a vertex coloring of $G$. If $\chi(H)=\omega(H)$ for every induced subgraph $H$ of $G$, then $G$ is called a {\it perfect graph}.
An {\it edge coloring} of $G$ is an assignment of colors to the edges of $G$ such that two adjacent edges receive different colors. The {\it chromatic index} of $G$, denoted by $\chi'(G)$, is the minimum number of colors required for an edge coloring of $G$.

An {\it independent set} in $G$ is a set of vertices such that no two of them are adjacent. The maximum size of an independent set in $G$, denoted by $\alpha(G)$, is called the {\it independence number} of $G$. A {\it matching} in $G$ is a set of edges such that no two of them are adjacent. The maximum size of a matching in $G$, denoted by $\alpha'(G)$, is called the {\it matching number} of $G$. A matching $M$ in $G$ is called a {\it perfect matching} if each vertex of $G$ is incident with some edge contained in $M$. A {\it vertex cover} in $G$ is a set of vertices that contains at least one endpoint of every edge. The minimum size of a vertex cover in $G$, denoted by $\beta(G)$, is called the {\it vertex covering number} of $G$.  An {\it edge cover} in $G$ is a set of edges such that every vertex of $G$ is incident with some edge contained in it. The minimum size of an edge cover in $G$, denoted by $\beta'(G)$, is called the {\it edge covering number} of $G$. A {\it dominating set} in $G$ is set $X$ of vertices such that every vertex of $V(G)\setminus X$ is adjacent with some vertex in $X$. The minimum size of a dominating set in $G$, denoted by $\gamma(G)$, is called the {\it domination number} of $G$.

\subsection{The proper divisor graph $\Upsilon_n$}\label{div-graph}

Let $n$ be a positive integer. An integer $d$ is called a {\it proper divisor} of $n$ if $1<d<n$ and $d$ divides $n$. The {\it proper divisor graph} of $n$, denoted by $\Upsilon_n$, is the graph with vertices the proper divisors of $n$, and two distinct vertices $u$ and $v$ are adjacent if and only if $n$ divides the product $uv$.

The graph $\Upsilon_n$ was recently introduced in the paper \cite{cps}. Note that $\Upsilon_n$ is the empty graph if and only if $n=1$ or $n$ is a prime. If $n$ is composite, then $\Upsilon_n$ is a connected graph by \cite[Lemma 2.6]{cps}.

\subsubsection*{Use of the graph $\Upsilon_n$:}

Let $G$ be a graph on $m$ vertices with $V(G)=\{v_1,v_2,\ldots ,v_m\}$ and $H_{1}, H_{2}, \ldots , H_{m}$ be $m$ pairwise vertex disjoint graphs. The {\it $G$-generalized join graph} of $H_{1}, H_2\ldots , H_{m}$ is the graph obtained from $G$ by replacing each vertex $v_i$ with the graph $H_{i}$ and then adding new edges from each vertex of $H_{i}$ to every vertex of $H_{j}$, $1\leq i\neq j\leq m$, whenever $v_i$ and $v_j$ are adjacent in $G$ (such graphs are called generalized composition graphs in \cite{schwenk}). Note that if $m=2$ and $G=K_2$, then the $G$-generalized join graph of $H_{1}$ and $H_{2}$ coincides with the usual join graph $H_{1}\vee H_{2}$ of $H_{1}$ and $H_{2}$.

The notion of zero divisor graph of a commutative ring was first introduced by I. Beck in \cite{beck} by taking all elements of the ring as vertices of the graph. It was later modified by Anderson and Livingston in \cite{anderson} as the following. The {\it zero divisor graph} $\Gamma(R)$ of a commutative ring $R$ with unity is the graph with vertex set consisting of the zero divisors of $R$, and  two distinct vertices $a$ and $b$ are adjacent if and only if $ab=0$ in $R$. Note that $\Gamma(R)$ is the empty graph if $R$ is an integral domain.

Let $n$ be composite. Since every proper divisor of $n$ is a zero divisor of the ring $\mathbb{Z}_n$ of integers modulo $n$, $\Upsilon_n$ is an induced subgraph of the zero divisor graph $\Gamma(\mathbb{Z}_{n})$. The graph $\Upsilon_n$ plays an important role in \cite{cps} while studying the spectrum of the Laplacian matrix of $\Gamma(\mathbb{Z}_{n})$. By \cite[Lemma 2.7]{cps}, $\Gamma(\mathbb{Z}_{n})$  is the $\Upsilon_n$-generalized join graph of certain complete graphs and null graphs corresponding to the proper divisors of $n$. It was proved in \cite[Proposition 4.1]{cps} that $\Gamma(\mathbb{Z}_{n})$ is Laplacian integral if and only if all the eigenvalues of the $m\times m$ vertex weighted Laplacian matrix $\mathbb{L}(\Upsilon_n)$ (defined in \cite[p.275]{cps}) of $\Upsilon_n$ are integers, where $m$ is the number of proper divisors of $n$. Further, by \cite[Theorem 5.8]{cps}, the algebraic connectivity of $\Gamma(\mathbb{Z}_{n})$ coincides with the second smallest eigenvalue of $\mathbb{L}(\Upsilon_n)$ if $n$ is not a prime power nor a product of two distinct primes, and the Laplacian spectral radius of $\Gamma(\mathbb{Z}_{n})$ coincides with the largest eigenvalue of $\mathbb{L}(\Upsilon_n)$ if $n$ is not a prime power. It is also known that $\Gamma(\mathbb{Z}_{n})$ is perfect if and only if $\Upsilon_n$ is prefect (see Section \ref{chromatic}).

One can refer to the book \cite{bapat} for different kinds of matrices associated with graphs and the papers \cite{akbari, anderson, ju, mathew} for more on the zero divisor graph of $\mathbb{Z}_{n}$.

\subsection{Aim of this paper}

It is clear from the discussion in Section \ref{div-graph} that the structure of the zero divisor graph $\Gamma(\mathbb{Z}_{n})$ is completely dependent on that of the proper divisor graph $\Upsilon_n$. In this paper, we shall mainly be interested in studying the automorphism group and different parameters of $\Upsilon_n$.

In Section \ref{basic}, we discuss some basic properties of $\Upsilon_n$ such as vertex degrees, pendant vertices, diameter and cut vertices. Similarity of two positive integers is defined in Section \ref{similarity}. For two composites $m$ and $n$, we prove that the proper divisor graphs $\Upsilon_m$ and $\Upsilon_n$ are isomorphic if and only if $m$ and $n$ are similar (except for distinct $m,n\in\{p_1^3, q_1q_2\}$, where $p_1,q_1,q_2$ are primes with $q_1\neq q_2$). We then determine the automorphism group of $\Upsilon_n$ in Section \ref{auto-group}. In Section \ref{parameters}, the graph parameters clique number, chromatic number, chromatic index, independence number, matching number, domination number, vertex and edge covering numbers of $\Upsilon_n$ are determined. We also provide algorithms for coloring the edges of $\Upsilon_n$ when $n$ is a prime power or a product of distinct primes.

\section{Basic properties of $\Upsilon_n$}\label{basic}

Let $n>1$ be an integer. The number of proper divisors of $n$ is denoted by $\pi(n)$. Let $k$ denote the number of distinct prime divisors of $n$ and
$$n=p_1^{\alpha_1}p_2^{\alpha_2}\cdots p_k^{\alpha_k}$$
be the prime power factorization of $n$, where $p_1,p_2,\ldots, p_k$ are distinct primes and $\alpha_1, \alpha_2,\ldots,\alpha_k$ are positive integers. We have
\begin{equation}\label{pi-n}
\pi(n)=(\alpha_1+1)(\alpha_2+1)\cdots(\alpha_k+1)-2,
\end{equation}
which follows from the paragraph before Lemma 2.6 in \cite{cps}.

\begin{lemma}\label{lem:elem}
Let $a,b$ be two positive integers. If $a$ is a proper divisor of $b$ with $\pi(a)+1=\pi(b)$, then $b=p^t$ and $a=p^{t-1}$ for some prime $p$ and some integer $t\geq 2$.
\end{lemma}

\begin{proof}
If $b$ has at least two distinct prime divisors, then the formula (\ref{pi-n}) implies that $\pi(b)-\pi(a)\geq 2$, which is not possible. So $b=p^t$ for some prime $p$ and some positive integer $t$. Since $b$ has a proper divisor $a$, we must have $t\geq 2$. Then, using (\ref{pi-n}), $\pi(a)+1=\pi(b)$ implies that $a=p^{t-1}$.
\end{proof}

In the rest of the paper, for obvious reason, we shall consider proper divisor graphs $\Upsilon_n$ only when $n$ is composite. If $n=p_1^2$, then $\Upsilon_n\cong K_1$. To avoid this triviality, if $n=p_1^{\alpha_1}$, then we shall also assume that $\alpha_1\geq 3$. We thus have $|V(\Upsilon_n)|=\pi(n)\geq 2$. Then connectedness of $\Upsilon_n$ implies that the degree of every vertex is at least one. For a positive integer $m$, we denote by $[m]:=\{1,2,\ldots,m\}$.

\subsection{Vertex degrees}

In the following proposition, we determine the degree of a vertex of $\Upsilon_n$ based on a divisibility condition involving $n$ and the square of that vertex.

\begin{proposition}\label{prop:deg}
Let $u$ be a vertex of $\Upsilon_n$. Then the degree of $u$ is given by:
$$\deg(u)=\begin{cases}
\pi(u) & \text{ if } n|u^2;\\
\pi(u)+1 & \text{ if } n\nmid u^2.
\end{cases}$$
\end{proposition}

\begin{proof}
If $v\sim u$ in $\Upsilon_n$, then $uv=rn$ for some integer $r\geq 1$ with $r\neq u$ and so $v=r\frac{n}{u}$. Since $v$ is a proper divisor of $n$, we get that $r$ divides $u$.
Conversely, if $r$ is a divisor of $u$ with $r\neq u$, then $w=r\frac{n}{u}$ is a proper divisor of $n$ and so is a vertex of $\Upsilon_n$. Also, $w\sim u$ if $w\neq u$.

Thus, the number of neighbours of $u$ is equal to the number of positive divisors $r$ of $u$ satisfying $r\neq u$ and $u\neq r\frac{n}{u}$. It follows that $\deg(u)=\pi(u)$ if $u=r\frac{n}{u}$ for some positive divisor $r$ of $u$ with $r\neq u$, otherwise $\deg(u)=\pi(u)+1$. Then the fact that $u=r\frac{n}{u}$ if and only if $n$ divides $u^2$ completes the proof.
\end{proof}

\begin{corollary}\label{coro:deg}
Let $n=p_1^{\alpha_1}p_2^{\alpha_2}\cdots p_k^{\alpha_k}$. If $\alpha_i=\alpha_j$ for some $i,j\in [k]$, then $\deg\left(\frac{n}{p_i}\right)= \deg\left(\frac{n}{p_j}\right)$.
\end{corollary}

\begin{proof}
Write $u_i=\frac{n}{p_i}$ and $u_j=\frac{n}{p_j}$. Since $\alpha_i=\alpha_j$, we have that $n$ divides $u_i^2$ if and only if $n$ divides $u_j^2$, and it follows from (\ref{pi-n}) that $\pi(u_i)=\pi(u_j)$. Then the corollary follows from Proposition \ref{prop:deg}.
\end{proof}

As an application of Proposition \ref{prop:deg}, all vertices of degree one and two in $\Upsilon_n$ are determined in the following.

\begin{proposition} \label{prop:pendant}
Let $n=p_1^{\alpha_1}p_2^{\alpha_2}\cdots p_k^{\alpha_k}$. Then the following hold:
\begin{enumerate}
\item[(i)] If $k=1$ and $\alpha_1\in\{3,4\}$, then $p_1$ and $p_1^2$ are the pendant vertices of $\Upsilon_n$.
\item[(ii)] If $k=1$ and $\alpha_1\geq 5$, then $p_1$ is the only pendant vertex of $\Upsilon_n$.
\item[(iii)] If $k\geq 2$, then $p_1,p_2,\ldots,p_k$ are precisely the pendant vertices of $\Upsilon_n$.
\end{enumerate}
\end{proposition}

\begin{proof}
The vertex $p_i$, $i\in[k]$, has exactly one neighbour in $\Upsilon_n$, namely $\frac{n}{p_i}$. So each of $p_1,p_2,\ldots,p_k$ is a pendant vertex of $\Upsilon_n$.

Now let $u$ be a vertex of $\Upsilon_n$ with $\deg(u)= 1$. By Proposition \ref{prop:deg}, we have $\pi(u)=0$ or $1$, and $n|u^2$ if the latter holds. If $\pi(u)=0$, then $u$ must be a prime. If $\pi(u)=1$, then $u$ must be the square of a prime. In that case, $n|u^2$ implies $n\in \{p_1^3,p_1^4\}$. It can be seen that $p_1^2$ is also a pendant vertex of $\Upsilon_n$ for $n\in \{p_1^3,p_1^4\}$.
\end{proof}

\begin{corollary}\label{coro:pendant}
Let $n=p_1^{\alpha_1}p_2^{\alpha_2}\cdots p_k^{\alpha_k}$ and $l$ be the number of pendant vertices of $\Upsilon_n$. Then
$$l=\begin{cases}
1 & \text{ if } n=p_1^{\alpha_1} \text{ with } \alpha_1 \geq 5;\\
2 & \text{ if } n\in\{p_1^3, p_1^4\};\\
k & \text{ if } k\geq 2.
\end{cases}$$
\end{corollary}

\begin{proposition}
Let $n=p_1^{\alpha_1}p_2^{\alpha_2}\cdots p_k^{\alpha_k}$ with $\alpha_1\geq \alpha_2\geq \cdots \geq \alpha_k$ and $V_2(n)$ be the set of all degree two vertices of $\Upsilon_n$. Then the following hold:
\begin{enumerate}
\item[(i)] $V_2(p_1^3)=\Phi$; $V_2(p_1^4)=\{p_1^3\}$; $V_2(n)=\{p_1^2, p_1^3\}$ if $n\in\{p_1^5,p_1^6\}$;
\item[(ii)] $V_2(p_1p_2)=\Phi$; $V_2(p_1^2p_2)=\{p_1^2, p_1p_2\}$; $V_2(p_1^2p_2^2)=\{p_1^2, p_2^2, p_1p_2\}$;
\item[(iii)] $V_2(n)=\{p_i^2: 1\leq i\leq k, \alpha_i\geq 2\}$ if $n\notin\{p_1^3,p_1^4,p_1^5,p_1^6, p_1p_2, p_1^2p_2,p_1^2p_2^2\}$.
\end{enumerate}
\end{proposition}

\begin{proof}
We shall prove (iii) only as the statements made in (i) and (ii) can easily be verified. So assume that $n\notin\{p_1^3,p_1^4,p_1^5,p_1^6, p_1p_2, p_1^2p_2,p_1^2p_2^2\}$. If $\alpha_i\geq 2$ for some $i\in [k]$, then $p_i^2$ is adjacent with the two vertices $\frac{n}{p_i},\; \frac{n}{p_i^2}$ only, and so $\deg(p_i^2)=2$.

Conversely, let $u$ be a vertex of $\Upsilon_n$ with $\deg(u)= 2$. If $n$ divides $u^2$, then Proposition \ref{prop:deg} gives that $\pi(u)=\deg(u)=2$, which is possible if and only if $u$ is the cube of a prime or the product of two distinct primes. This forces $n\in\{p_1^4,p_1^5,p_1^6, p_1^2p_2,p_1^2p_2^2\}$ (as $u$ is a proper divisor of $n$), a contradiction. So $n\nmid u^2$. Then $\pi(u)=1$ by Proposition \ref{prop:deg} again, which is possible if and only if $u$ is the square of a prime. Therefore, $V_2(n)=\{p_i^2: 1\leq i\leq k, \alpha_i\geq 2\}$.
\end{proof}

\begin{proposition}\label{prop:degcomp}
Let $n=p_1^{\alpha_1}p_2^{\alpha_2}\cdots p_k^{\alpha_k}$ and $u, v$ be two distinct vertices of $\Upsilon_n$. If $u$ divides $v$, then $\deg(u)\leq \deg(v)$, and equality holds if and only if $k=1,$ $u=p_1^{\lceil \alpha_1/2 \rceil-1}$ and $v=p_1^{\lceil \alpha_1/2 \rceil}.$
\end{proposition}

\begin{proof}
Since $u$ is a proper divisor of $v$, we have $\pi(u)+1\leq \pi(v)$. Then, using Proposition \ref{prop:deg}, we get
\begin{equation}\label{eqn-1}
\deg(u)\leq \pi(u)+1\leq \pi(v)\leq \deg(v).
\end{equation}

If $k=1$, $u=p_1^{\lceil \alpha_1/2 \rceil-1}$ and $v=p_1^{\lceil \alpha_1/2 \rceil}$, then observe that $n\nmid u^2$ and $n|v^2$. Applying Proposition \ref{prop:deg}, we get $\deg(u)=\pi(u)+1=\lceil \alpha_1/2 \rceil-1=\pi(v)=\deg(v)$. Conversely, suppose that $\deg(u)= \deg(v)$. Then (\ref{eqn-1}) gives
$$\deg(u)=\pi(u)+1= \pi(v)=\deg(v).$$
It follows that both $u$ and $v$ are powers of a prime by Lemma \ref{lem:elem}, and that $n\nmid u^2$ and $n|v^2$ by Proposition \ref{prop:deg}. Thus we have the following:
\begin{enumerate}
\item[(i)] $n$ itself is a prime power and so $n=p_1^{\alpha_1}$.
\item[(ii)] $u=p_1^{t-1}$ and $v=p_1^{t}$ for some integer $t$ with $2\leq t\leq \alpha_1-1$ (Lemma \ref{lem:elem}).
\item[(iii)] $p_1^{\lceil \alpha_1/2 \rceil}\nmid u$ as $n\nmid u^2$, and $p_1^{\lceil \alpha_1/2 \rceil}| v$ as $n\mid v^2$.
\end{enumerate}
It follows from (ii) and (iii) that $t-1\leq \lceil \alpha_1/2 \rceil -1$ and $t\geq \lceil \alpha_1/2 \rceil$. This gives $t=\lceil \alpha_1/2 \rceil$, thus completing the proof.
\end{proof}

\begin{remark}\label{remark}
Recall that a graph is called nearly irregular if it contains exactly one pair of vertices with equal degree. Let $G$ be a connected nearly irregular graph on $m\geq 2$ vertices. If $u$ and $v$ are the two vertices of $G$ with equal degree, then $\deg(u)=\deg(v)=\lceil\frac{m+1}{2}\rceil -1$. This follows from Proposition \ref{prop:degcomp} with $n=p_1^{m+1}$ and the fact that there is precisely one connected nearly irregular graph on $m$ vertices up to isomorphism (\cite[Theorem 1.12]{CZ}).
\end{remark}

\subsection{Diameter}

The following proposition determines the diameter of $\Upsilon_n$ which is denoted by $\diam(\Upsilon_n)$.

\begin{proposition}
Let $n=p_1^{\alpha_1}p_2^{\alpha_2}\cdots p_k^{\alpha_k}$. Then
$$\diam(\Upsilon_n)=
\begin{cases}
1 & \text{ if } n\in\{p_1^3, p_1p_2\};\\
2 & \text{ if } n=p_1^{\alpha_1} \text{ with } \alpha_1\geq 4;\\
3 & \text{otherwise.}
\end{cases}$$
\end{proposition}

\begin{proof}
It follows from the proof of connectedness of $\Upsilon_n$ in \cite[Lemma 2.6]{cps} that $\diam(\Upsilon_n)\leq 3$. If $n\in\{p_1^3, p_1p_2\}$, then $\Upsilon_n\cong K_2$ and so $\diam(\Upsilon_n)=1$. Suppose that $n=p_1^{\alpha_1}$ with $\alpha_1\geq 4$. The vertex $p_1^{\alpha_1-1}$ is adjacent with all other vertices of $\Upsilon_n$. Further, $\alpha_1\geq 4$ implies that $p_1\sim p_1^{\alpha_1-1}\sim p_1^2$ is the shortest path between $p_1$ and $p_1^2$. Hence $\diam(\Upsilon_n)=2$.

Now consider $k\geq 2$ with $n\neq p_1p_2$. By Proposition \ref{prop:pendant}(iii), $p_1$ and $p_2$ are pendant vertices of $\Upsilon_n$. We have $p_1\sim \frac{n}{p_1}$ and $p_2\sim \frac{n}{p_2}$. The vertices $p_1,p_2, \frac{n}{p_1}, \frac{n}{p_2}$ are pairwise distinct and $p_1\sim \frac{n}{p_1}\sim \frac{n}{p_2}\sim p_2$ is the shortest path between $p_1$ and $p_2$.  Therefore, $\diam(\Upsilon_n)=3$.
\end{proof}

\subsection{Minimum and maximum degrees}

By Corollary \ref{coro:pendant}, $\Upsilon_n$ has at least one pendant vertex and hence the minimum degree of $\Upsilon_n$ is one. The vertices of maximum degree in $\Upsilon_n$ are determined in Proposition \ref{prop:maxdeg} below. We need the following lemma.

\begin{lemma}\label{lem:maxdeg}
Let $n=p_1^{\alpha_1}p_2^{\alpha_2}\cdots p_k^{\alpha_k}$ with $k\geq 2$ and $\alpha_1\geq \alpha_2\geq \cdots \geq \alpha_k$. Set $u_i:=\frac{n}{p_i}$ for $i\in [k]$. Then $\deg(u_1)\geq \deg(u_j)$ for $2\leq j \leq k$, and equality holds if and only if $\alpha_1=\alpha_j$ or $n=p_1^2p_2$.
\end{lemma}

\begin{proof}
We have $\alpha_1\geq \alpha_j$. If $\alpha_1=\alpha_j$, then $\deg(u_1)=\deg(u_j)$ by Corollary \ref{coro:deg}. Assume that $\alpha_1>\alpha_j$. Then using Proposition \ref{prop:deg} and the inequality $(\alpha_1+1)\alpha_j < \alpha_1(\alpha_j+1)$, we get that
\begin{align}
\deg(u_j)\leq \pi(u_j)+1&=(\alpha_1+1)\cdots(\alpha_{j-1}+1)\alpha_j(\alpha_{j+1}+1)\cdots(\alpha_k+1)-1\nonumber \\
&\leq \alpha_1(\alpha_2+1)\cdots(\alpha_j+1)\cdots(\alpha_k+1)-2\label{inequ-1}\\
&=\pi(u_1)\leq \deg(u_1).\nonumber
\end{align}
If $k\geq 3$, then the inequality (\ref{inequ-1}) is strict and this gives $\deg(u_1)> \deg(u_j)$. Suppose that $\deg(u_1)= \deg(u_j)$ with $\alpha_1 >\alpha_j$. Then $k=2$, $j=2$ and we must have
$$\deg(u_2)=\pi(u_2)+1=(\alpha_1+1)\alpha_2 -1 =  \alpha_1(\alpha_2+1) -2 =\pi(u_1)=\deg(u_1).$$
It follows that $n\nmid u_2^2$ by Proposition \ref{prop:deg} and that $\alpha_1=\alpha_2 +1$. The former implies that $\alpha_2=1$ and so $\alpha_1=2$. This given $n=p_1^2p_2$. If $n=p_1^2p_2$, then it can be seen that $\deg(u_1)=2= \deg(u_2)$.
\end{proof}

\begin{proposition}\label{prop:maxdeg}
Let $n=p_1^{\alpha_1}p_2^{\alpha_2}\cdots p_k^{\alpha_k}$ with $\alpha_1\geq \alpha_2\geq \cdots \geq \alpha_k$. Then the following hold:
\begin{enumerate}
\item[(i)] If $n=p_1^3$, then both the vertices $p_1$ and $p_1^2$ of $\Upsilon_n$ are of maximum degree.
\item[(ii)] If $n=p_1^2p_2$, then $\frac{n}{p_1}=p_1p_2$ and $\frac{n}{p_2}=p_1^2$ are the vertices of $\Upsilon_n$ of maximum degree.
\item[(iii)] If $n\notin\{p_1^3,p_1^2p_2\}$, then $\frac{n}{p_1},\ldots, \frac{n}{p_t}$ are precisely the vertices of $\Upsilon_n$ of maximum degree, where $t\in [k]$ is the largest integer such that $\alpha_t=\alpha_1$.
\end{enumerate}
\end{proposition}

\begin{proof}
The statements made in (i) and (ii) can easily be verified. We shall prove (iii). If $n=p_1^{\alpha_1}$ with $\alpha_1\geq 4$, then $\frac{n}{p_1}$ is the only vertex that is adjacent with all other vertices of $\Upsilon_n$ and so it is the only vertex of $\Upsilon_n$ of maximum degree.

Now consider $k\geq 2$. Let $u\in V(\Upsilon_n)$ be of maximum degree. Set $u_i:=\frac{n}{p_i}$ for $i\in [k]$. Since $u$ is a divisor of $u_j$ for some $j\in [k]$, we have $\deg(u)\leq \deg(u_j)$ by Proposition \ref{prop:degcomp}. The maximality of $\deg(u)$ then gives that $\deg(u)= \deg(u_j)$. Since $k\geq 2$, Proposition \ref{prop:degcomp} again implies that $u=u_j$.
Thus, the vertices of $\Upsilon_n$ of maximum degree are contained in $\{u_1,u_2,\ldots,u_k\}$. By the definition of $t$, we have $\alpha_1=\alpha_2=\cdots=\alpha_t$ and $\alpha_1>\alpha_j$ for $t+1\leq j\leq k$. Since $n\neq p_1^2p_2$, Lemma \ref{lem:maxdeg} gives that $u_1,u_2,\ldots,u_t$ are precisely the vertices of maximum degree in $\Upsilon_n$.
\end{proof}

The number of vertices of $\Upsilon_n$ with minimum degree one (that is, pendant vertices) is given in Corollary \ref{coro:pendant}. As a consequence of Proposition \ref{prop:maxdeg}, we have the following result on the number of vertices of $\Upsilon_n$ having maximum degree.

\begin{corollary}\label{cor:maxdegcount}
Let $n=p_1^{\alpha_1}p_2^{\alpha_2}\cdots p_k^{\alpha_k}$ with $\alpha_1\geq \alpha_2\geq \cdots \geq \alpha_k$ and $L$ be the number of vertices of $\Upsilon_n$ with maximum degree. If $t\in [k]$ is the largest integer with $\alpha_t=\alpha_1$, then
$$L=\begin{cases}
2 & \text{ if }n\in\{p_1^3, p_1^2p_2\};\\
t & \text{ otherwise. }
\end{cases}$$
\end{corollary}

\begin{corollary}\label{coro:maxdeg}
Let $n=p_1^{\alpha_1}p_2^{\alpha_2}\cdots p_k^{\alpha_k}$ with $\alpha_1\geq \alpha_2\geq \cdots \geq \alpha_k$. Then the maximum degree $\Delta\left(\Upsilon_n\right)$ of $\Upsilon_n$ is given by:
$$\Delta\left(\Upsilon_n\right)=\deg\left(\frac{n}{p_1}\right)=\begin{cases}
\pi\left(\frac{n}{p_1}\right) & \text{ if } k=1, \text{ or } k\geq 2 \text{ with }  \alpha_1 \geq 2;\\
\pi\left(\frac{n}{p_1}\right)+1 & \text{ otherwise. }
\end{cases}$$
\end{corollary}

\begin{proof}
This follows from Propositions \ref{prop:deg} and \ref{prop:maxdeg}.
\end{proof}

\subsection{Cut vertices}

Let $v$ be a vertex of $\Upsilon_n$. For two disjoint nonempty subsets $A$ and $B$ of $V(\Upsilon_n -v)=V(\Upsilon_n)\setminus\{v\}$, we say that $(A,B)$ is a {\it separation} of $\Upsilon_n -v$ if $V(\Upsilon_n -v) =A\cup B$ and there is no edge of $\Upsilon_n -v$ with one endpoint in $A$ and the other in $B$. Since $\Upsilon_n$ is connected, we have that $v$ is a cut vertex of $\Upsilon_n$ if and only if there exists a separation of $\Upsilon_n -v$.

If $n\in \{p_1^3, p_1p_2\}$, then $\Upsilon_n\cong K_2$ has no cut vertex. We shall find the cut vertices of $\Upsilon_n$ for the remaining values of $n$.

\begin{proposition}
Let $n=p_1^{\alpha_1}p_2^{\alpha_2}\cdots p_k^{\alpha_k}$ with $n\notin\{p_1^3, p_1p_2\}$. Then $\frac{n}{p_1}, \frac{n}{p_2},\ldots, \frac{n}{p_k}$ are precisely the cut vertices of $\Upsilon_n$.
\end{proposition}

\begin{proof}
Each $p_i$, $i\in [k]$, is a pendant vertex of $\Upsilon_n$ by Proposition \ref{prop:pendant} (with $p_i\sim \frac{n}{p_i}$). Since $n\notin\{p_1^3, p_1p_2\}$, we have $\pi(n)=|V(\Upsilon_n)|\geq 3$ and then connectedness of $\Upsilon_n$ implies that each $\frac{n}{p_i}$ is a cut vertex. We claim that any cut vertex of $\Upsilon_n$ is one of $\frac{n}{p_1}, \frac{n}{p_2},\ldots, \frac{n}{p_k}$.

If $k=1$, then $\frac{n}{p_1}$ is adjacent with all other vertices of $\Upsilon_n$ and so it must be the only cut vertex of $\Upsilon_n$.
Now consider $k\geq 2$. Suppose that there exists a cut vertex $v$ of $\Upsilon_n$ different from $\frac{n}{p_1}, \frac{n}{p_2},\ldots, \frac{n}{p_k}$. Let $(A,B)$ be a separation of $\Upsilon_n-v$. Since $\frac{n}{p_i}\sim\frac{n}{p_j}$ for distinct $i,j\in [k]$, all the vertices $\frac{n}{p_i}$, $i\in [k]$, are either in $A$ or in $B$. Without loss of generality, we may assume that $\frac{n}{p_i}\in A$ for all $i\in [k]$. Since $B$ is nonempty, there is a vertex $u\in B$. Since $u$ is divisible by $p_j$ for some $j\in [k]$, we must have $u\sim \frac{n}{p_j}$, contradicting that $(A,B)$ is a separation of $\Upsilon_n-v$.
\end{proof}

\section{Similarity and isomorphisms} \label{similarity}

Let $n$ and $m$ be two positive integers with their prime power factorisations:
$$n=p_1^{\alpha_1}p_2^{\alpha_2}\cdots p_k^{\alpha_k}\quad \text{ and } \quad m=q_1^{\beta_1}q_2^{\beta_2}\cdots q_l^{\beta_l},$$
where $p_i, q_j$ are primes and $\alpha_i,\beta_j$ are positive integers with $\alpha_1\geq \alpha_2\geq \cdots \geq \alpha_k$ and $\beta_1\geq \beta_2\geq \cdots \geq \beta_l$. We say that $n$ and $m$ are {\it similar} if $k=l$ and $\alpha_i=\beta_i$ for every $i\in [k]$.

If two composites $n$ and $m$ are similar, then the proper divisor graphs $\Upsilon_n$ and $\Upsilon_m$ are isomorphic. This can be seen from the fact that the construction of the graph $\Upsilon_n$ does not depend on the actual primes involved as divisors of $n$. What about the converse statement? If $n=p_1^3$ and $m=q_1q_2$, then $\Upsilon_n \cong K_2\cong \Upsilon_m$, but $n$ and $m$ are not similar. We prove that the converse statement is also true with this particular example as the only exception. More precisely, we have the following:

\begin{theorem}\label{thm:sim}
Let $m$ and $n$ be composite integers. Then $\Upsilon_m$ and $\Upsilon_n$ are isomorphic if and only if $m$ and $n$ are similar, except for $m,n\in\{p_1^3, q_1q_2\}$ with $m\neq n$, where $p_1,q_1,q_2$ are primes with $q_1\neq q_2$.
\end{theorem}

The proof of Theorem \ref{thm:sim} follows from Propositions \ref{lem:l2} and \ref{prop:sim} below.

\begin{proposition}\label{lem:l2}
Let $n=p_1^{\alpha_1}$ with $\alpha_1 \geq 4$. If $\Upsilon_n\cong \Upsilon_m$ for some composite $m$, then $n$ and $m$ are similar.
\end{proposition}

\begin{proof}
We have $|V(\Upsilon_n)|=|V(\Upsilon_m)|$ as $\Upsilon_n\cong \Upsilon_m$. If $m=q_1^{\beta_1}$ for some prime $q_1$ and positive integer $\beta_1$, then $\alpha_1-1=\pi(n)=|V(\Upsilon_n)|=|V(\Upsilon_m)|=\pi(m)=\beta_1-1$ gives $\beta_1=\alpha_1$.

Therefore, it is enough to prove that $m$ is a prime power. This is true if $\alpha_1=4$, as there is no integer $m$ with at least two distinct prime divisors for which $\pi(m)=|V(\Upsilon_m)|=|V(\Upsilon_n)|=\alpha_1 -1=3$. If $\alpha_1\geq 5$, then the claim follows from Proposition \ref{prop:pendant} and the fact that $\Upsilon_n$ and $\Upsilon_m$ must have the same number of pendant vertices.
\end{proof}

\begin{proposition}\label{prop:sim}
Let $n=p_1^{\alpha_1}p_2^{\alpha_2}\cdots p_k^{\alpha_k}$ and $m=q_1^{\beta_1}q_2^{\beta_2}\cdots q_l^{\beta_l}$ with $k\geq 2$, $l\geq 2$, $\alpha_1\geq \alpha_2\geq \cdots \geq \alpha_k$ and $\beta_1\geq \beta_2\geq \cdots \geq \beta_l$. If $\Upsilon_n\cong \Upsilon_m$, then $n$ and $m$ are similar.
\end{proposition}

\begin{proof}
By Proposition \ref{prop:pendant}(iii), $p_1,p_2,\ldots,p_k$ are the pendant vertices of $\Upsilon_n$ and $q_1,q_2,\ldots,q_l$ are that of $\Upsilon_m$. Since $\Upsilon_n\cong \Upsilon_m$, they have the same number of pendant vertices and hence $k=l$. The fact that $|V(\Upsilon_n)|=|V(\Upsilon_m)|$ gives
\begin{equation}\label{eq:pinpim}
(\alpha_1+1)(\alpha_2+1)\cdots(\alpha_k+1)=(\beta_1+1)(\beta_2+1)\cdots(\beta_k+1)
\end{equation}
Using (\ref{eq:pinpim}), it can be seen that $n=p_1^2p_2$ if and only if $m=q_1^2q_2$, and so $n,m$ are similar in this case. Therefore, we shall assume that $n\neq p_1^2p_2$ (and hence $m\neq q_1^2q_2$).

Let $\phi:\Upsilon_n\rightarrow \Upsilon_m$ be a graph isomorphism. Since $\phi$ maps pendant vertices of $\Upsilon_n$ to that of $\Upsilon_m$, it induces a bijection from $\{p_1,p_2,\ldots,p_k\}$ to $\{q_1,q_2,\ldots,q_k\}$. In order to prove the proposition, it is enough to show that $\alpha_i=\beta_s$ if $\phi(p_i)=q_s$ for $i,s\in [k]$.

The only neighbour of $p_i$ in $\Upsilon_n$ is $\frac{n}{p_i}$ and that of $q_s$ in $\Upsilon_m$ is $\frac{m}{q_s}$. Since $\phi(p_i)=q_s$, we must have $\phi\left(\frac{n}{p_i}\right)=\frac{m}{q_s}$ and hence
\begin{equation}\label{eq:deg1deg2}
\deg\left(\frac{n}{p_i}\right)=\deg\left(\frac{m}{q_s}\right).
\end{equation}
Note that $n$ divides $\left(\frac{n}{p_i}\right)^2$ if and only if $\alpha_i\geq 2$, and $m$ divides $\left(\frac{m}{q_s}\right)^2$ if and only if $\beta_s\geq 2$. So we have the following by Proposition \ref{prop:deg}:
\begin{align}\label{eq:deg1}
\alpha_i= 1: \quad\quad &\deg\left(\frac{n}{p_i}\right)=(\alpha_1+1)\cdots(\alpha_{i-1}+1)\alpha_i(\alpha_{i+1}+1)\cdots(\alpha_k+1)-1;
\end{align}
\begin{align}\label{eq:deg2}
\alpha_i\geq 2: \quad\quad &\deg\left(\frac{n}{p_i}\right)=(\alpha_1+1)\cdots(\alpha_{i-1}+1)\alpha_i(\alpha_{i+1}+1)\cdots(\alpha_k+1)-2;
\end{align}
\begin{align}\label{eq:deg3}
\beta_s=1: \quad\quad & \deg\left(\frac{m}{q_s}\right)=(\beta_1+1)\cdots(\beta_{s-1}+1)\beta_s(\beta_{s+1}+1)\cdots(\beta_k+1)-1;
\end{align}
\begin{align}\label{eq:deg4}
\beta_s\geq 2: \quad\quad &\deg\left(\frac{m}{q_s}\right)=(\beta_1+1)\cdots(\beta_{s-1}+1)\beta_s(\beta_{s+1}+1)\cdots(\beta_k+1)-2.
\end{align}
If $\alpha_i\geq 2$ and $\beta_s\geq 2$, then equations (\ref{eq:pinpim}), (\ref{eq:deg1deg2}), (\ref{eq:deg2}) and (\ref{eq:deg4}) give $\frac{\alpha_i+1}{\alpha_i}=\frac{\beta_s+1}{\beta_s}$, that is, $\alpha_i=\beta_s$.

Now suppose that $\alpha_i=1$. If $\beta_s=1$, then we are done. Suppose that $\beta_s\geq 2$. We shall get a contradiction by showing that $m=q_1^2q_2$. Putting $\alpha_i=1$ in (\ref{eq:pinpim}), we get
\begin{equation}\label{eq:deg5}
(\alpha_1+1)\cdots (\alpha_{i-1}+1)(\alpha_{i+1}+1)\cdots(\alpha_k+1)=\frac{1}{2}(\beta_1+1)\cdots(\beta_s+1)\cdots(\beta_k+1).
\end{equation}
From (\ref{eq:deg1deg2}), (\ref{eq:deg1}) and (\ref{eq:deg4}), we get
\begin{align}\label{eq:deg6}
(\alpha_1+1)\cdots (\alpha_{i-1}+1)(\alpha_{i+1}+1)\cdots(\alpha_k+1)&= (\beta_1+1)\cdots(\beta_{s-1}+1)\beta_s \nonumber\\
 &\quad\quad\quad (\beta_{s+1}+1)\cdots(\beta_k+1)-1.
\end{align}
An easy calculation using the equations (\ref{eq:deg5}) and (\ref{eq:deg6}) gives that
\begin{equation}\label{eq:deg7}
(\beta_s-1)\prod^{k}_{\substack{t=1\\ t\neq s}}(\beta_t+1)=2.
\end{equation}
Since $k\geq 2$, $\beta_s\geq 2$ and $\beta_t\geq 1$ for $t\neq s$, it follows from (\ref{eq:deg7}) that $k=2$, $\beta_s=2$ and $\beta_t=1$, where $\{s,t\}=\{1,2\}$. Since $\beta_s>\beta_t$, we must have $s=1, t=2$ and hence $m=q_1^2q_2$.

Similarly, if $\alpha_i\geq 2$ and $\beta_s=1$, then we shall get a contradiction by showing that $n= p_1^2p_2$ (using the equations (\ref{eq:pinpim}), (\ref{eq:deg1deg2}), (\ref{eq:deg2}) and (\ref{eq:deg3})).
\end{proof}

From the proof of Proposition \ref{prop:sim}, we have the following result which is useful in determining the automorphism group of $\Upsilon_n$ in the next section.

\begin{corollary}\label{coro:sim}
Let $n=p_1^{\alpha_1}p_2^{\alpha_2}\cdots p_k^{\alpha_k}$ with $k\geq 2$ and $\alpha_1\geq \alpha_2\geq \cdots \geq \alpha_k$. If $n\neq p_1^2p_2$ and $\phi:\Upsilon_n\rightarrow \Upsilon_n$ is a graph automorphism, then $\phi$ permutes the pendant vertices $p_1,p_2,\ldots,p_k$ of $\Upsilon_n$ such that $\phi(p_i)= p_j$ implies $\alpha_i=\alpha_j$ for $1\leq i,j\leq k$.
\end{corollary}

\section{The automorphism group $\Aut(\Upsilon_n)$}\label{auto-group}

Let $n=p_1^{\alpha_1}p_2^{\alpha_2}\cdots p_k^{\alpha_k}$ with $\alpha_1\geq \alpha_2\geq \cdots \geq \alpha_k$. We study the automorphism group $\Aut(\Upsilon_n)$ of $\Upsilon_n$ in this section. We have the following result when $k=1$.

\begin{proposition}
If $n=p_1^{\alpha_1}$, then $|\Aut(\Upsilon_n)|=2$.
\end{proposition}

\begin{proof}
Recall that $\alpha_1\geq 3$ by our assumption. By Proposition \ref{prop:degcomp}, the degrees of the vertices of $\Upsilon_n$ are pairwise distinct, with the exception of two vertices $u:=p_1^{\lceil \frac{\alpha_1}{2}\rceil -1}$ and $v:=p_1^{\lceil \frac{\alpha_1}{2}\rceil}$ for which the degrees are the same. Therefore, every automorphism of $\Upsilon_n$ must fix each of the vertices contained in $V(\Upsilon_n)\setminus\{u,v\}$.

The map $\phi:\Upsilon_n\rightarrow \Upsilon_n$ with $\phi(u)=v$, $\phi(v)=u$ and which is identity on $V(\Upsilon_n)\setminus\{u,v\}$ is an automorphism of $\Upsilon_n$. This follows from the fact that a vertex of $V(\Upsilon_n)\setminus\{u,v\}$ is adjacent with either both $u$ and $v$, or none of them (note that $u\sim v$ if and only if $\alpha_1$ is odd). Therefore, $|\Aut(\Upsilon_n)|=2$.
\end{proof}

If $k\geq 2$, then $p_1,p_2,\ldots,p_k$ are precisely the pendant vertices of $\Upsilon_n$ (Proposition \ref{prop:pendant}(iii)). Therefore, for $\phi\in \Aut(\Upsilon_n)$, $\phi(p_i)$ is also a prime for every $i\in [k]$.

\begin{lemma}\label{lem:aut-1}
Let $k\geq 2$ and $\phi$ be an automorphism of $\Upsilon_n$. Then the following hold:
\begin{enumerate}
\item[(i)] $\phi\left(\frac{n}{p_i}\right)=\frac{n}{\phi(p_i)}$ for $1\leq i\leq k$.
\item[(ii)] If $n\neq p_1^2p_2$ and $w=p_1^{s_1}p_2^{s_2}\cdots p_k^{s_k}\in V(\Upsilon_n)$, then
$\phi(w)=\phi(p_1)^{t_1}\phi(p_2)^{t_2}\cdots \phi(p_k)^{t_k},$
where $t_i\geq 1$ if and only $s_i\geq 1$ for $1\leq i\leq k$.
\end{enumerate}
\end{lemma}

\begin{proof}
(i) Since $p_i$ and $\phi(p_i)$ both are pendant vertices with $p_i\sim\frac{n}{p_i}$ and $\phi(p_i)\sim\frac{n}{\phi(p_i)}$, we must have $\phi\left(\frac{n}{p_i}\right)=\frac{n}{\phi(p_i)}$ for $1\leq i\leq k$.

(ii) Let $u,v$ be vertices of $\Upsilon_n$ such that $\phi(u)=v$. We claim that if $p_i$ divides $u$, then $p_j=\phi(p_i)$ divides $v$. Assume that $p_i|u$. If $u\neq \frac{n}{p_i}$, then $u \sim \frac{n}{p_i}$ implies that $v=\phi(u)\sim \phi\left(\frac{n}{p_i}\right)=\frac{n}{\phi(p_i)}=\frac{n}{p_j}$ and hence $p_j$ divides $v$. If $u=\frac{n}{p_i}$, then $p_i^2$ divides $n$ and so $\alpha_i\geq 2$. Since $n\neq p_1^2p_2$, we have $\alpha_j=\alpha_i\geq 2$ by Corollary \ref{coro:sim}. Then $v=\phi\left(\frac{n}{p_i}\right)=\frac{n}{\phi(p_i)}=\frac{n}{p_j}$ is divisible by $p_j$. Applying similar argument to the automorphism $\phi^{-1}$ of $\Upsilon_n$ and using the fact that $\phi^{-1}(v)=u$, we get that if $p_l$ divides $v$, then the prime $\phi^{-1}(p_l)$ divides $u$.

Taking $u=w=p_1^{s_1}p_2^{s_2}\cdots p_k^{s_k}$, it follows from the previous paragraph that $\phi(w)=\phi(p_1)^{t_1}\phi(p_2)^{t_2}\cdots \phi(p_k)^{t_k}$, where $t_i\geq 1$ if and only $s_i\geq 1$ for $1\leq i\leq k$.
\end{proof}

\begin{lemma}\label{lem:aut-2}
Let $k\geq 2$ with $n\neq p_1^2p_2$ and $\phi$ be an automorphism of $\Upsilon_n$. Then $\phi\left(p_i^{s_i}\right)=\phi(p_i)^{s_i}$ and $\phi\left(\frac{n}{p_i^{s_i}}\right)=\frac{n}{\phi(p_i)^{s_i}}$ for $1\leq i\leq k$ and $1\leq s_i\leq \alpha_i$.
\end{lemma}

\begin{proof}
Let $u=p_i^{s_i}$. By Lemma \ref{lem:aut-1}(ii), we have $\phi(u)=\phi(p_i)^{t_i}$ for some positive integer $t_i$. Since $k\geq 2$, Proposition \ref{prop:deg} gives that $\deg(u)=\pi(u)+1=s_i$ and $\deg(\phi(u))=\pi(\phi(u))+1=t_i$. Then $\deg(u)=\deg(\phi(u))$ gives that $s_i=t_i$ and hence $\phi(u)=\phi(p_i)^{s_i}$. The second part that $\phi\left(\frac{n}{p_i^{s_i}}\right)=\frac{n}{\phi(p_i)^{s_i}}$ can be obtained from the following.

Let $p_j=\phi(p_i)$. Then $\alpha_i=\alpha_j$ by Corollary \ref{coro:sim}. We claim that $\phi\left(\frac{n}{p_i^{l}}\right)=\frac{n}{p_j^{l}}$ for $1\leq l\leq \alpha_i$. The proof is by induction on $l$. If $l=1$, then the claim follows from Lemma \ref{lem:aut-1}(i). Assume that $\phi\left(\frac{n}{p_i^{l}}\right)=\frac{n}{p_j^{l}}$ for $1\leq l\leq m<\alpha_i$. The neighbours of $p_i^{m+1}$ are precisely $\frac{n}{p_i}, \frac{n}{p_i^2},\ldots, \frac{n}{p_i^{m+1}}$ and that of $\phi\left(p_i^{m+1}\right)=\phi(p_i)^{m+1}=p_j^{m+1}$ are $\frac{n}{p_j}, \frac{n}{p_j^2},\ldots, \frac{n}{p_j^{m+1}}$. Since $\phi$ is one-one and the neighbours of $p_i^{m+1}$ are mapped to the neighbours of $\phi\left(p_i^{m+1}\right)$, the induction hypothesis then implies that $\phi\left(\frac{n}{p_i^{m+1}}\right)=\frac{n}{p_j^{m+1}}$. This proves the claim.
\end{proof}

\begin{proposition}\label{prop:aut}
Let $k\geq 2$ with $n\neq p_1^2p_2$ and $\phi$ be an automorphism of $\Upsilon_n$. Then for any vertex $p_1^{s_1}p_2^{s_2}\cdots p_k^{s_k}$ of $\Upsilon_n$, we have
$$\phi\left(p_1^{s_1}p_2^{s_2}\cdots p_k^{s_k}\right)=\phi(p_1)^{s_1}\phi(p_2)^{s_2}\cdots \phi(p_k)^{s_k}.$$
\end{proposition}

\begin{proof}
Let $u=p_1^{s_1}p_2^{s_2}\cdots p_k^{s_k}$. By Lemma \ref{lem:aut-1}(ii), we have $\phi(u)=\phi(p_1)^{t_1}\phi(p_2)^{t_2}\cdots \phi(p_k)^{t_k}$, where $t_i\geq 1$ if and only $s_i\geq 1$ for $1\leq i\leq k$.
Suppose that $s_i\neq t_i$ for some $i$. Then $s_i\geq 1$ and $t_i\geq 1$. Let $\phi(p_i)=p_j$. Then $\alpha_i=\alpha_j$ by Corollary \ref{coro:sim}.

{\bf Claim 1:} $u\neq \frac{n}{p_i^{s_i}}$. If possible, suppose that $u=\frac{n}{p_i^{s_i}}$. Then $\alpha_i=2s_i$. By Lemma \ref{lem:aut-2}, we have $\phi(u)=\phi\left(\frac{n}{p_i^{s_i}}\right)=\frac{n}{p_j^{s_i}}$. This gives $s_i+t_i=\alpha_j$. Then $\alpha_j=\alpha_i=2s_i$ implies that $s_i=t_i$, contradicting our assumption.

{\bf Claim 2:} $\phi(u)\neq \frac{n}{\phi(p_i)^{t_i}}$. If possible, suppose that $\phi(u)= \frac{n}{\phi(p_i)^{t_i}}$. Then $\alpha_j=2t_i$. Since $\phi\left(\frac{n}{p_i^{t_i}}\right)=\frac{n}{\phi(p_i)^{t_i}}$ by Lemma \ref{lem:aut-2}, injectivity of $\phi$ gives $u=\frac{n}{p_i^{t_i}}$. This implies $s_i+t_i=\alpha_i$. Then $\alpha_i=\alpha_j$ gives that $s_i=t_i$, contradicting our assumption.

Since $u \sim \frac{n}{p_i^{s_i}}$, we have $\phi(u)\sim \phi\left(\frac{n}{p_i^{s_i}}\right)=\frac{n}{\phi\left(p_i\right)^{s_i}}$ (Lemma \ref{lem:aut-2}). This implies $s_i< t_i$ (as $s_i\neq t_i$).
Since $\phi(u)\sim \frac{n}{\phi(p_i)^{t_i}}=\phi\left(\frac{n}{p_i^{t_i}}\right)$, we get $u\sim \frac{n}{p_i^{t_i}}$. But this is not possible as $s_i< t_i$. Therefore, $s_i=t_i$ for all $i\in [k]$ and hence $\phi(u)=\phi(p_1)^{s_1}\phi(p_2)^{s_2}\cdots \phi(p_k)^{s_k}$.
\end{proof}

\begin{corollary}
If $k\geq 2$ and  $\alpha_1> \alpha_2> \cdots > \alpha_k\geq 1$, then $|\Aut(\Upsilon_n)|=2$ or $1$ according as $n=p_1^2p_2$ or not.
\end{corollary}

\begin{proof}
If $n=p_1^2p_2$, then $\Upsilon_n$ is a path of length three and hence $|\Aut(\Upsilon_n)|=2$. So assume that $n\neq p_1^2p_2$. Let $\phi$ be an automorphism of $\Upsilon_n$. Since $\alpha_1> \alpha_2> \cdots > \alpha_k$, $\phi$ fixes each of the pendant vertices $p_1,p_2,\ldots,p_k$ by Corollary \ref{coro:sim}. Then Proposition \ref{prop:aut} implies that each vertex of $\Upsilon_n$ is fixed by $\phi$ and so $\phi$ is the identity map.
\end{proof}

\begin{corollary}\label{coro-aut-1}
If $k\geq 2$ and two automorphisms of $\Upsilon_n$ agree on the pendant vertices, then they are equal.
\end{corollary}

\begin{lemma}\label{lem:aut-3}
Let $k\geq 2$ and $\alpha_{i_1}=\alpha_{i_2}=\cdots =\alpha_{i_a}$ for some subset $A:=\{i_1,i_2,\ldots,i_a\}$ of $[k]$. Given a permutation $\tau$ of $\{p_{i_1},p_{i_2},\ldots,p_{i_a}\}$, define the map $\overline{\tau}:\Upsilon_n\rightarrow \Upsilon_n$ by
$$\overline{\tau}\left(p_1^{s_1}p_2^{s_2}\cdots p_k^{s_k}\right)=\left(\underset{i_j\in A}\prod \tau(p_{i_j})^{s_{i_j}}\right)\left(\underset{l\in [k]\setminus A}\prod p_{l}^{s_{l}}\right)$$
for $p_1^{s_1}p_2^{s_2}\cdots p_k^{s_k}\in V(\Upsilon_n)$. Then $\overline{\tau}$ is an automorphism of $\Upsilon_n$.
\end{lemma}

\begin{proof}
If $\tau(p_{i_j})=p_{i_r}$ for $i_j,i_r\in A$, then $s_{i_j}\leq \alpha_{i_j}=\alpha_{i_r}$ and so $\overline{\tau}$ is well-defined.

Let $u=p_1^{s_1}p_2^{s_2}\cdots p_k^{s_k}$ and $v= p_1^{t_1}p_2^{t_2}\cdots p_k^{t_k}$ be two vertices of $\Upsilon_n$. If $\overline{\tau}\left(u\right)=\overline{\tau}\left(v\right)$, then comparing the prime powers on both sides we get $s_{i_j}=t_{i_j}$ for $i_j\in A$ and $s_l=t_l$ for $l\in [k]\setminus A$. So $s_r=t_r$ for $r\in[k]$ and hence $\overline{\tau}$ is injective. Write $v=\left(\underset{i_j\in A}\prod p_{i_j}^{t_{i_j}}\right)\left(\underset{l\in [k]\setminus A}\prod p_{l}^{t_{l}}\right)$and define
$$w:=\left(\underset{i_j\in A}\prod \tau^{-1}(p_{i_j})^{t_{i_j}}\right)\left(\underset{l\in [k]\setminus A}\prod p_{l}^{t_{l}}\right).$$
If $\tau^{-1}(p_{i_j})=p_{i_r}$ for some $i_r\in A$, then $t_{i_j}\leq \alpha_{i_j}=\alpha_{i_r}$ and so $w$ is a vertex of $\Upsilon_n$. Now it is clear that $\overline{\tau}(w)=v$, implying $\overline{\tau}$ is surjective.

Since $\alpha_{i_1}=\cdots =\alpha_{i_a}$, it can be observed that $\overline{\tau}(u)\sim \overline{\tau}(v)$ if and only if $s_r+t_r\geq \alpha_r$ for all $r\in[k]$. The later holds if and only if $u\sim v$. Hence $u\sim v$ if and only if $\overline{\tau}(u)\sim \overline{\tau}(v)$. Thus $\overline{\tau}$ is an automorphism of $\Upsilon_n$.
\end{proof}

In the following proposition, we determine the full automorphism group of $\Upsilon_n$ when $k\geq 2$ with $n\neq p_1^2p_2$.
Let $\alpha_{r_1},\alpha_{r_2},\ldots,\alpha_{r_b}$ be the distinct integers in the list $\alpha_1, \alpha_2, \ldots, \alpha_k$. For $1\leq i\leq b$, define
$$A_{r_i}:=\{j\in [k]: \alpha_j=\alpha_{r_i}\}$$
and set $|A_{r_i}|=k_i$. Then $A_{r_1}\cup A_{r_2}\cup\cdots\cup A_{r_b}$ is a partition of $[k]$ and so $k_1+k_2+\cdots +k_b=k$. For a given positive integer $m$, $S_m$ denotes the symmetric group defined on the set $[m]$.

\begin{proposition}\label{prop:aut-full}
Let $k\geq 2$ with $n\neq p_1^2p_2$. Then $\Aut(\Upsilon_n)\cong S_{k_1}\times S_{k_2}\times \cdots \times S_{k_b}$, where the integers $k_1,k_2,\ldots,k_b$ are as defined above.
\end{proposition}

\begin{proof}
For $1\leq i\leq b$, consider the sets $A_{r_i}$ as defined above and let $X_{r_i}:=\{p_j: j\in A_{r_i}\}$. Then $X_{r_1}\cup X_{r_2}\cup\cdots\cup X_{r_b}$ is a partition of the set $X=\{p_1,p_2,\ldots, p_k\}$. Given $\phi\in \Aut(\Upsilon_n)$, let $\phi_{r_i}$ denote the restriction of $\phi$ to $X_{r_i}$. Then $\phi_{r_i}$ is a permutation of $X_{r_i}$ by Corollary \ref{coro:sim}. This gives $(\phi_{r_1},\ldots,\phi_{r_b})\in Sym(X_{r_1})\times \cdots \times Sym(X_{r_b})$, where $Sym(X_{r_i})$ is the symmetric group defined on $X_{r_i}$. Thus the map $f:\Aut(\Upsilon_n)\rightarrow Sym(X_{r_1})\times \cdots \times Sym(X_{r_b})$ taking $\phi$ to $(\phi_{r_1},\phi_{r_2},\ldots,\phi_{r_b})$ is well defined. We prove that $f$ is a group isomorphism.

Let $\phi,\psi\in \Aut(\Upsilon_n)$. We claim that $f(\phi\psi)=f(\phi)f(\psi)$. It is enough to show that $(\phi\psi)_{r_i}=\phi_{r_i}\psi_{r_i}$ for $1\leq i\leq b$. Indeed, for $p_j\in X_{r_i}$, we have $$(\phi\psi)_{r_i}(p_j)=(\phi\psi)(p_j)=\phi(\psi(p_j))=\phi(\psi_{r_i}(p_j))=\phi_{r_i}(\psi_{r_i}(p_j))=(\phi_{r_i}\psi_{r_i})(p_j).$$
Thus $f$ is a group homomorphism. Corollary \ref{coro-aut-1} implies that $f$ is injective. Consider $(\tau_1,\tau_2,\ldots,\tau_b)\in Sym(X_{r_1})\times \cdots \times Sym(X_{r_b})$. For $1\leq i\leq b$, let $\overline{\tau}_i$ be the automorphism of $\Upsilon_n$ as obtained in Lemma \ref{lem:aut-3}. Define $\overline{\tau}:=\overline{\tau}_1 \overline{\tau}_2\cdots\overline{\tau}_b$, the composition of $\overline{\tau}_1, \overline{\tau}_2,\ldots,\overline{\tau}_b$. Then $\overline{\tau}\in \Aut(\Upsilon_n)$ and observe that $\overline{\tau}_{r_i}=\tau_i$ for each $i\in [b]$. This gives $f(\overline{\tau})=(\tau_1,\tau_2,\ldots,\tau_b)$ and hence $f$ is surjective.

Thus $f$ is a group isomorphism and so $\Aut(\Upsilon_n)\cong Sym(X_{r_1})\times \cdots \times Sym(X_{r_b})$. Since $|X_{r_i}|=|A_{r_i}|=k_i$, we have $Sym(X_{r_i})\cong S_{k_i}$ for every $i\in [b]$ and hence the result follows.
\end{proof}

\section{Graph parameters of $\Upsilon_n$}\label{parameters}

In this section, we shall determine the graph parameters clique number, chromatic number, chromatic index, domination number,  independence number, matching number, vertex and edge covering numbers of $\Upsilon_n$.

\subsection{Clique number $\omega(\Upsilon_n)$}\label{subsec:clique}

Let $n=p_1^{\alpha_1}p_2^{\alpha_2}\cdots p_k^{\alpha_k}$ be the prime power factorization of $n$ such that the integers $\alpha_1,\ldots,\alpha_l$ are odd and $\alpha_{l+1},\ldots, \alpha_k$ are even for some $l\in\{0,1,2,\ldots,k\}$. Consider the subsets $A$ and $B$ of $V(\Upsilon_n)$ as defined below:
\begin{align*}
A:=&\left\{p_1^{r_1}p_2^{r_2}\cdots p_k^{r_k}: \lceil \alpha_i/2\rceil \leq r_i \leq \alpha_i, 1\leq i \leq k \right\}\setminus\{n\},\\
B:= & \left\{\frac{n}{p_j^{\lceil \alpha_j/2\rceil}}: 1\leq j \leq l\right\},
\end{align*}
where $B=\Phi$ if $l=0$. Let $\mathcal{K}$ denote the induced subgraph of $\Upsilon_n$ with vertex set $A\cup B$. Observe that $A$ and $B$ are disjoint, and that any two distinct vertices in $A\cup B$ are adjacent. Thus $\mathcal{K}$ is a clique in $\Upsilon_n$ with
\begin{equation}\label{eq:clique1}
|V(\mathcal{K})|=|A|+|B|=\left\lceil \frac{\alpha_1}{2}\right\rceil \left\lceil \frac{\alpha_2}{2}\right\rceil \cdots \left\lceil \frac{\alpha_l}{2}\right\rceil \left(\frac{\alpha_{l+1}}{2}+1\right) \cdots \left(\frac{\alpha_k}{2}+1\right) +l-1.
\end{equation}

\begin{proposition}\label{prop:clique}
Let $n=p_1^{\alpha_1}p_2^{\alpha_2}\cdots p_k^{\alpha_k}$ be the prime power factorization of $n$ such that the integers $\alpha_1,\ldots,\alpha_l$ are odd and $\alpha_{l+1},\ldots, \alpha_k$ are even for some $l\in\{0,1,2,\ldots,k\}$. Then
$$\omega(\Upsilon_n)= \left\lceil \frac{\alpha_1}{2}\right\rceil \left\lceil \frac{\alpha_2}{2}\right\rceil\cdots \left\lceil \frac{\alpha_l}{2}\right\rceil \left(\frac{\alpha_{l+1}}{2}+1\right) \cdots \left(\frac{\alpha_k}{2}+1\right) +l-1.$$
\end{proposition}

\begin{proof}
Consider the clique $\mathcal{K}$ in $\Upsilon_n$ defined above with vertex set $A\cup B$. Let $H$ be an arbitrary clique in $\Upsilon_n$. We prove that $|V(H)|\leq |V(\mathcal{K})|$ and then (\ref{eq:clique1}) would complete the proof. It is enough to show that there exists an injective map $\phi: V(H)\rightarrow V(\mathcal{K})$.

Let $y=p_1^{s_1}\cdots p_l^{s_l}p_{l+1}^{s_{l+1}}\cdots p_k^{s_k}$ be a vertex of $H$.
If $\lceil \alpha_i/2\rceil \leq s_i \leq \alpha_i$ for every $i\in [k]$, then $y$ is a vertex of $\mathcal{K}$ that is contained in $A$. In this case, we define $\phi(y):=y\in A$.

Suppose that $s_i < \lceil \frac{\alpha_i}{2} \rceil$ for some $i\in [k]$. Let $j\in [k]$ be the smallest integer such that $s_j < \lceil \frac{\alpha_j}{2} \rceil$.
For every vertex $z=p_1^{t_1}\cdots p_l^{t_l}p_{l+1}^{t_{l+1}}\cdots p_k^{t_k}\in V(H)\setminus\{y\}$, the fact that $y\sim z$ gives
$$t_j\geq
\begin{cases}
\lceil \frac{\alpha_j}{2} \rceil & \text{ if } j\leq l;\\
\lceil \frac{\alpha_j}{2}\rceil +1= \frac{\alpha_j}{2}+1 & \text{ if } j\geq l+1.
\end{cases}$$
Thus, $y$ is the only vertex of $H$ with $s_j < \lceil \frac{\alpha_j}{2} \rceil$. If $j\leq l$, we define
$$\phi(y):=\frac{n}{p_j^{\lceil \alpha_j/2\rceil}}\in B.$$
If $j\geq l+1$, then there is no such vertex $z$ of $H$ with $t_j=\lceil \frac{\alpha_j}{2}\rceil=\frac{\alpha_j}{2}$. In this case, we define
$$\phi(y):=p_1^{\alpha_1}\cdots p_{j-1}^{\alpha_{j-1}}p_j^{\alpha_j/2}p_{j+1}^{\alpha_{j+1}}\cdots p_k^{\alpha_k}\in A.$$
It follows from the construction of the map $\phi: V(H)\rightarrow V(\mathcal{K})$ that $\phi$ is well-defined and it is one-one.
\end{proof}

As a consequence of Proposition \ref{prop:clique}, we have the following:

\begin{corollary}\label{coro:clique}
Let $n=p_1^{\alpha_1}p_2^{\alpha_2}\cdots p_k^{\alpha_k}$. Then $\omega(\Upsilon_n)\geq k$.
\end{corollary}

Corollary \ref{coro:clique} can also be seen directly as follows. Since $\frac{n}{p_i}\sim \frac{n}{p_j}$ for $1\leq i\neq j\leq k$, the induced subgraph of $\Upsilon_n$ with vertex set $\left\{\frac{n}{p_i}: 1\leq i \leq k\right\}$ is a clique in $\Upsilon_n$.

\subsection{Chromatic number $\chi(\Upsilon_n)$}\label{chromatic}

In the following proposition, we prove that the chromatic number and the clique number of $\Upsilon_n$ are equal.

\begin{proposition}\label{prop:chro}
Let $n=p_1^{\alpha_1}p_2^{\alpha_2}\cdots p_k^{\alpha_k}$ be the prime power factorization of $n$ such that the integers $\alpha_1,\ldots,\alpha_l$ are odd and $\alpha_{l+1},\ldots, \alpha_k$ are even for some $l\in\{0,1,2,\ldots,k\}$. Then
$$\chi(\Upsilon_n)=\omega(\Upsilon_n)=\left\lceil \frac{\alpha_1}{2}\right\rceil \left\lceil \frac{\alpha_2}{2}\right\rceil\cdots \left\lceil \frac{\alpha_l}{2}\right\rceil \left(\frac{\alpha_{l+1}}{2}+1\right) \cdots \left(\frac{\alpha_k}{2}+1\right) +l-1.$$
\end{proposition}

\begin{proof}
Consider the clique $\mathcal{K}$ in $\Upsilon_n$ (defined in Section \ref{subsec:clique} with $V(\mathcal{K})=A\cup B$). We have $\omega(\Upsilon_n)=|V(\mathcal{K})|$ by (\ref{eq:clique1}) and Proposition \ref{prop:clique}. Assign $\omega(\Upsilon_n)$ distinct colors to the vertices of $\mathcal{K}$. Out of the $\omega(\Upsilon_n)$ colors used so far, we shall choose $k$ of them (possible as $\omega(\Upsilon_n)\geq k$ by Corollary \ref{coro:clique}) and assign these $k$ colors suitably to the remaining vertices of $\Upsilon_n$.

For $1\leq i\leq k$, set $\gamma_i:= \lceil \frac{\alpha_i}{2} \rceil$ and let $c_i$ be the color assigned to the vertex $w_i\in V(\mathcal{K})$, where
$$w_i=\begin{cases}
p_1^{\alpha_1}\cdots p_{i-1}^{\alpha_{i-1}}p_i^{\gamma_i-1}p_{i+1}^{\alpha_{i+1}}\cdots p_l^{\alpha_l} \cdots p_k^{\alpha_k} & \text{ if } 1\leq i\leq l;\\
p_1^{\alpha_1}\cdots p_l^{\alpha_l}\cdots p_{i-1}^{\alpha_{i-1}}p_i^{\gamma_i}p_{i+1}^{\alpha_{i+1}} \cdots p_k^{\alpha_k} & \text{ if } l+1\leq i\leq k.
\end{cases}$$
Let $u$ be a vertex of $\Upsilon_n$ outside $\mathcal{K}$. Then $p_i^{\gamma_i}$ does not divide $u$ for some $i\in [k]$. We assign the color $c_t$ to the vertex $u$, where $t\in[k]$ is the smallest integer such that $p_t^{\gamma_t}\nmid u$. Thus $u$ and $w_t$ receive the same color. In this way, we color all the vertices of $\Upsilon_n$. Note that if $x, y$ are two distinct vertices of $\Upsilon_n$ with the same color $c_t$, then $p_t^{\alpha_t}\nmid xy$ implies $x\nsim y$.
\end{proof}

Since $\chi(\Upsilon_n)=\omega(\Upsilon_n)$, it is natural to ask whether $\Upsilon_n$ is perfect. In \cite{Smith}, the {\it zero-divisor type graph} $\Gamma^T(\mathbb{Z}_n)$ of $\mathbb{Z}_n$ is defined and it is proved in Theorem 4.1 that $\Gamma^T(\mathbb{Z}_n)$ is perfect if and only if the zero divisor graph $\Gamma(\mathbb{Z}_n)$ is perfect. Further, using the Strong Perfect Graph Theorem, the author proved that the graph $\Gamma^T(\mathbb{Z}_n)$ is perfect if and only if $n\in\{p_1^{\alpha_1}, p_1^{\alpha_1}p_2^{\alpha_2}, p_1^{\alpha_1}p_2 p_3, p_1p_2p_3p_4\}$. From the construct of $\Gamma^T(\mathbb{Z}_n)$, it can be seen that the proper divisor graph $\Upsilon_n$ is isomorphic to $\Gamma^T(\mathbb{Z}_n)$. As a consequence, it follows that $\Upsilon_n$ is prefect if and only if $n\in\{p_1^{\alpha_1}, p_1^{\alpha_1}p_2^{\alpha_2}, p_1^{\alpha_1}p_2 p_3, p_1p_2p_3p_4\}$.

\subsection{Matching number $\alpha'(\Upsilon_n)$}

\begin{proposition}\label{prop:match}
Let $n=p_1^{\alpha_1}p_2^{\alpha_2}\cdots p_k^{\alpha_k}$ and $\mathcal{M}$ be the collection of all edges of $\Upsilon_n$ of the form $\left\{x,\frac{n}{x}\right\}$ with $x <\sqrt{n}$. Then $\mathcal{M}$ is a matching in $\Upsilon_n$ of maximum size and
$$\alpha'(\Upsilon_n)=|\mathcal{M}|=\lfloor \pi(n)/ 2\rfloor.$$
\end{proposition}

\begin{proof}
Since $|V(\Upsilon_n)|=\pi(n)$, we have $\alpha'(\Upsilon_n)\leq \lfloor\frac{\pi(n)}{2}\rfloor$. Clearly, two distinct edges contained in $\mathcal{M}$ do not share any common vertex. So $\mathcal{M}$ is a matching in $\Upsilon_n$. Every vertex of $\Upsilon_n$ is an end point of some edge contained in $\mathcal{M}$, with the exception of the vertex $\sqrt{n}$ when $n$ is a perfect square (in which case $|V(\Upsilon_n)|$ is odd). This gives $|\mathcal{M}|=\lfloor\frac{\pi(n)}{2}\rfloor$ and it follows that the matching $\mathcal{M}$ is of maximum size.
\end{proof}

\begin{corollary}
$\Upsilon_n$ has a perfect matching if and only if $n$ is not a perfect square.
\end{corollary}

\begin{corollary}\label{coro-2:match}
Let $Z=\{x\in V(\Upsilon_n): x <\sqrt{n}\}$. Then $|Z|= \lfloor\frac{\pi(n)}{2}\rfloor$.
\end{corollary}

\begin{proof}
This follows from Proposition \ref{prop:match} using the fact that $Z$ is in bijective correspondence with the set $\mathcal{M}:=\left\{\left\{x,\frac{n}{x}\right\}: x\in V(\Upsilon_n),\;x <\sqrt{n}\right\}$.
\end{proof}

\begin{corollary}
The edge covering number $\beta'(\Upsilon_n)$ of $\Upsilon_n$ is given by: $\beta'(\Upsilon_n)=\lceil \pi(n)/2\rceil$.
\end{corollary}

\begin{proof}
Since $\Upsilon_n$ has no isolated vertices, we have $\alpha'(\Upsilon_n)+\beta'(\Upsilon_n)=|V(\Upsilon_n)|=\pi(n)$ by \cite[Theorem 3.1.22]{west}. Then Proposition \ref{prop:match} gives that $\beta'(\Upsilon_n)=\pi(n)-\lfloor \pi(n)/2\rfloor=\lceil \pi(n)/2\rceil$.
\end{proof}

\subsection{Independence number $\alpha(\Upsilon_n)$}

If $n=p_1p_2$, then $\Upsilon_n\cong K_2$ and so $\alpha(\Upsilon_n)=1$. If $n\neq p_1p_2$, then no two vertices among $p_1,p_2,\ldots,p_k$ are adjacent and hence $\alpha(\Upsilon_n)\geq k$. We prove the following:

\begin{proposition}\label{prop:indepent}
Let $n=p_1^{\alpha_1}p_2^{\alpha_2}\cdots p_k^{\alpha_k}$ and $\mathcal{I}:=\{x\in V(\Upsilon_n): x\leq \sqrt{n}\}$. Then $\mathcal{I}$ is an independent set in $\Upsilon_n$ of maximum size and
$$\alpha(\Upsilon_n)=|\mathcal{I}|=\lceil \pi(n)/2\rceil.$$
\end{proposition}

\begin{proof}
If $x,y\in\mathcal{I}$ with $x\neq y$, then at least one of them is less than $\sqrt{n}$ and so  $xy<n$. This implies that $x$ and $y$ are not adjacent. Thus $\mathcal{I}$ is an independent set in $\Upsilon_n$.

Let $\mathcal{J}$ be an independent set in $\Upsilon_n$ of maximum size. Then $|\mathcal{I}|\leq |\mathcal{J}|$. We claim that $|\mathcal{I}|\geq |\mathcal{J}|$. This follows if $\mathcal{J}\setminus \mathcal{I}=\Phi$. Assume that $\mathcal{J}\setminus \mathcal{I}\neq\Phi$.  Let $y\in \mathcal{J}\setminus \mathcal{I}$. Then $y>\sqrt{n}$. This implies $\frac{n}{y}< \sqrt{n}$ and so $\frac{n}{y}\in \mathcal{I}$. Since $y\sim \frac{n}{y}$ (as $n\neq y^2$) and $\mathcal{J}$ is an independent set, it follows that $\frac{n}{y}\notin \mathcal{J}$ and hence $\frac{n}{y}\in \mathcal{I}\setminus\mathcal{J}$. Thus $y\mapsto \frac{n}{y}$ defines an injective map from $\mathcal{J}\setminus \mathcal{I}$ to $\mathcal{I}\setminus \mathcal{J}$. Then $|\mathcal{J}|=|\mathcal{I}\cap\mathcal{J}|+|\mathcal{J}\setminus \mathcal{I}|\leq |\mathcal{I}\cap\mathcal{J}|+|\mathcal{I}\setminus \mathcal{J}|=|\mathcal{I}|$. Thus $|\mathcal{I}|= |\mathcal{J}|$ and hence the independent set $\mathcal{I}$ is of maximum size.

We have $\sqrt{n}\in\mathcal{I}$ if and only if $n$ is a perfect square if and only if $\pi(n)=|V(\Upsilon_n)|$ is odd. Consider the set $Z$ defined in Corollary \ref{coro-2:match}. When $n$ is a perfect square, we have
$Z=\mathcal{I}\setminus\{\sqrt{n}\}$ and this gives $|\mathcal{I}|=|Z|+1=\lfloor\frac{\pi(n)}{2}\rfloor+1=\lceil\frac{\pi(n)}{2}\rceil$, otherwise $Z=\mathcal{I}$ and we get $|\mathcal{I}|=|Z|=\lfloor\frac{\pi(n)}{2}\rfloor=\lceil\frac{\pi(n)}{2}\rceil$.
\end{proof}

\begin{corollary}
The vertex covering number $\beta(\Upsilon_n)$ of $\Upsilon_n$ is given by: $\beta(\Upsilon_n)=\lfloor \pi(n)/2\rfloor$.
\end{corollary}

\begin{proof}
We have $\alpha(\Upsilon_n)+\beta(\Upsilon_n)=|V(\Upsilon_n)|=\pi(n)$ by \cite[Lemma 3.1.21]{west}. Then Proposition \ref{prop:indepent} gives that $\beta(\Upsilon_n)=\pi(n)-\lceil \pi(n)/2\rceil=\lfloor \pi(n)/2\rfloor$.
\end{proof}

\subsection{Domination number $\gamma(\Upsilon_n)$}

\begin{proposition}
Let $n=p_1^{\alpha_1}p_2^{\alpha_2}\cdots p_k^{\alpha_k}$. If $n\neq p_1p_2$, then $Y:=\left\{\frac{n}{p_1},\frac{n}{p_2}\ldots, \frac{n}{p_k}\right\}$ is a dominating set of minimum size and so
$$\gamma(\Upsilon_n)=\begin{cases}
1 & \text{ if } n=p_1p_2;\\
k & \text{ otherwise. }
\end{cases}$$
\end{proposition}

\begin{proof}
Every vertex of $\Upsilon_n$ is adjacent or equal to at least one of the vertices in $Y$. This implies that $Y$ is a dominating set in $\Upsilon_n$. Assume that $n\neq p_1p_2$. Then the set consisting of the $k$ distinct edges $\left\{p_i,\frac{n}{p_i}\right\}$, $1 \leq i\leq k$, is a matching in $\Upsilon_n$. The fact that $p_1,p_2,\ldots, p_k$ are pendant vertices of $\Upsilon_n$ then implies that any dominating set in $\Upsilon_n$ must contain $p_i$ or $\frac{n}{p_i}$ for every $i\in [k]$. Thus every dominating set must contain at least $k$ vertices and hence the dominating set $Y$ is of minimum size. The rest is clear.
\end{proof}

\subsection{Chromatic index $\chi'(\Upsilon_n)$}

Clearly, $\chi'(\Upsilon_n)\geq \Delta(\Upsilon_n)$. We shall prove that $\chi'(\Upsilon_n)=\Delta(\Upsilon_n)$. The following result proved in \cite[Remark 1]{akbari} is helpful for us, which was obtained as an application of Vizing's Adjacency Lemma \cite[Corollary 3.6(iii)]{yap}.

\begin{lemma}\cite{akbari}\label{lem:chindex}
Let $G$ be a simple graph and $D:=\{u\in V(G):\deg(u)=\Delta(G)\}$. Suppose that, for every vertex $u\in D$, there is an edge $\{u,w\}$ of $G$ such that $\Delta(G)-\deg(w)+2>|D|$.
Then $\chi'(G)=\Delta(G)$.
\end{lemma}

\begin{proposition}\label{prop:chro-inex}
Let $n=p_1^{\alpha_1}p_2^{\alpha_2}\cdots p_k^{\alpha_k}$. Then $\chi'(\Upsilon_n)=\Delta(\Upsilon_n)$.
\end{proposition}

\begin{proof}
If $n=p_1^3$ or $p_1p_2$, then $\Upsilon_n\cong K_2$ and so $\chi'(\Upsilon_n)=1=\Delta(\Upsilon_n)$. Assume that $n\notin\{p_1^3, p_1p_2\}$. Define $u_i:=\frac{n}{p_i}$ for $1\leq i\leq k$ and $D:=\{v\in V(\Upsilon_n):\deg(v)=\Delta(\Upsilon_n)\}$.  By Proposition \ref{prop:maxdeg}, we have $D\subseteq \{u_1,u_2,\ldots,u_k\}$ and so $1\leq |D|\leq k$.

Let $v\in D$. Then $v= u_i$ for some $i\in [k]$. We have $u_i\sim p_i$ and $u_i\sim u_j$ for $j\in [k]\setminus\{i\}$. Since $n\neq p_1p_2$, the vertex $p_i$ is different from each such $u_j$. Thus $\deg(v)=\deg(u_i)\geq k$. For the edge $\{v,p_i\}=\{u_i,p_i\}$ of $\Upsilon_n$, we have
$\Delta(\Upsilon_n)-\deg(p_i)+2=\deg(v)-1+2 \geq k+1>|D|.$
Hence $\chi'(\Upsilon_n)=\Delta(\Upsilon_n)$ by Lemma \ref{lem:chindex}.
\end{proof}

\subsubsection*{Coloring the edges of $\Upsilon_n$:}

When $n$ is a prime power or a product of distinct primes, we are able to provide algorithms in the following for a proper edge coloring of $\Upsilon_n$ using $\Delta(\Upsilon_n)$ distinct colors. It would be interesting to develop such algorithms for other values of $n$ as well.

\subsubsection*{(i) $n=p_1^{\alpha_1}$:}

We have $\chi'(\Upsilon_n)=\Delta(\Upsilon_n)=\pi\left(\frac{n}{p_1}\right)=\alpha_1-2$ by Proposition \ref{prop:chro-inex} and Corollary \ref{coro:maxdeg}. Consider $\alpha_1-2$ distinct colours, say $c_0,c_1,\ldots,c_{\alpha_1-3}$. For a given edge $\{u,v\}$ of $\Upsilon_n$, there exists a unique $j\in\{0,1,\ldots,\alpha_1-3\}$ such that $\frac{uv}{n}=p_1^j$. Then assign the color $c_j$ to the edge $\{u,v\}$. If two distinct adjacent edges $\{u,v\}$ and $\{v, w\}$ of $\Upsilon_n$ receive the same color $c_t$, then $\frac{uv}{n}=p_1^t=\frac{vw}{n}$ gives $u=w$, a contradiction. Thus, two distinct adjacent edges receive different colors, implying a proper edge coloring of $\Upsilon_n$.

\subsubsection*{(ii) $n=p_1p_2\cdots p_k$, $k\geq 2$:}

Here $\chi'(\Upsilon_n)=\Delta(\Upsilon_n)=\pi\left(\frac{n}{p_1}\right)+1=2^{k-1}-1$ by Proposition \ref{prop:chro-inex} and Corollary \ref{coro:maxdeg}. Consider the colors $c_j$ for $j\in I$, where $I$ is the index set defined by
$$I:=\{j : j|n \text{ and } 1<j< \sqrt{n}\}.$$
We thus have a total of $|I|=\pi(n)/2=2^{k-1}-1$ colors (see Corollary \ref{coro-2:match}). We assign colors to the edges of $\Upsilon_n$ in the following way. Let $\{u,v\}$ be an edge of $\Upsilon_n$. Set $l=\frac{uv}{n}$. Then $1\leq l < n$ and $l|n$. We have $l\neq\sqrt{n}$ as $n$ is not a perfect square. If $l\neq 1$, then $l\in I$ if $l<\sqrt{n}$, and $\frac{n}{l}\in I$ if $l> \sqrt{n}$.
\begin{itemize}
\item If $1< l < \sqrt{n}$, then we call $\{u,v\}$ an edge of Type-I and  assign the color $c_l$ to it.
\item If $\sqrt{n}<l<n$, then we call $\{u,v\}$ an edge of Type-II and assign the color $c_{\frac{n}{l}}$ to it.
\item If $l=1$, then we call $\{u,v\}$ an edge of Type-III.
\end{itemize}
We first prove that there is no conflict of interest in the above assignment of colors to Type-I and Type II edges, and then we shall color Type-III edges suitably.

Let $\{u,v\}$ and $\{v,w\}$ be distinct adjacent edges of $\Upsilon_n$ of Type-I/Type-II. Suppose that they have received the same color.
If both $\{u,v\}$ and $\{v,w\}$ are of the same type, then it follows that $\frac{uv}{n}=\frac{vw}{n}$ which gives $u=w$, a contradiction. So assume that they are of different types. Without loss of generality, we may assume that $\{u,v\}$ is of Type-I and $\{v,w\}$ is of Type-II. Then it follows that
$\frac{uv}{n}=\frac{n^2}{vw}$ and this gives $uv^2w=n^3$. Since $n$ is a product of distinct primes and $v\neq n$, we have that $p_j$ does not divide $v$ for some $j\in [k]$. Then $uv^2w$ is not divisible by $p_j^3$. This implies that there are no such $u,v, w$ with $uv^2w=n^3$. Thus, two distinct adjacent edges of Type-I/Type-II have received different colors.

Now consider an edge $\{x,y\}$ of Type-III. Then $xy=n$. Note that any other edge that is adjacent with $\{x,y\}$ must be of Type-I or Type-II. Let $a=\deg(x)-1$ and $b=\deg(y)-1$. If we prove that $a+b < 2^{k-1}-1$, then we can color the edge $\{x,y\}$ with any of the $2^{k-1}-1-(a+b)$ colors which are not used for the Type-I/Type-II edges adjacent with $\{x,y\}$.

If one of $x$ or $y$, say $x$, is a pendant vertex, then $\deg(y)=\Delta(\Upsilon_n)$ by Propositions \ref{prop:pendant}(iii) and \ref{prop:maxdeg}(iii). This gives $a=0$ and $b=\Delta(\Upsilon_n)-1=2^{k-1}-2$ and so the claim follows. Assume that none of $x,y$ is a pendant vertex. Then none of $x,y$ is a prime. Since $n=xy$ and $n$ is a product of distinct primes, there exist distinct primes $p_i,p_j$ each dividing $x$ but not $y$, and there exist distinct primes $p_s, p_t$ each dividing $y$ but not $x$. Since $\deg(x)=\pi(x)+1\leq 2^{k-2}-1$, we have $a\leq 2^{k-2}-2$. Similarly, $b\leq 2^{k-2}-2$. It follows that $a+b\leq 2^{k-1}-4< 2^{k-1}-1$.

\hskip 1cm

\noindent {\bf Addresses:}\\
\noindent {\bf Hitesh Kumar, Kamal Lochan Patra, Binod Kumar Sahoo}
\begin{enumerate}
\item[1)] School of Mathematical Sciences, National Institute of Science Education and Research (NISER), Bhubaneswar, P.O.- Jatni, District- Khurda, Odisha-752050, India

\item[2)] Homi Bhabha National Institute (HBNI), Training School Complex, Anushakti Nagar, Mumbai-400094, India
\end{enumerate}
\noindent E-mails: hitesh.kumar@niser.ac.in, klpatra@niser.ac.in, bksahoo@niser.ac.in
\end{document}